\documentclass[10pt, conference]{IEEEtran}

\pdfoutput=1

\usepackage[numberinsection]{prelude}
\usepackage[referable]{threeparttablex}

\newtoggle{long}
\toggletrue{long}

\setlist{nosep}
\makeatletter
\renewlist{tablenotes}{enumerate}{1}
\setlist[tablenotes]{
  label=\tnote{\alph*},
  ref=\alph*,
  itemsep=\z@,
  topsep=\smallskipamount,
  partopsep=\z@skip,
  parsep=\z@,
  itemindent=10pt,
  labelindent=*,
  labelsep=\z@,
  leftmargin=\z@,
  align=left,
  before={\footnotesize}
}
\makeatother

\newcounter{priorproof}
\cCrefname{priorproof}{Prior Proof}{Prior Proofs}
\newenvironment{priorproof}[1]{%
  \refstepcounter{priorproof}
  \begin{namedobservation*}{Prior Proof~\thepriorproof}
    #1.
    \begin{description}[font={\normalfont\itshape}, nosep, beginpenalty=10000]
}{%
    \end{description}
  \end{namedobservation*}
  \ignorespacesafterend%
}

\makeatletter
\let\oldtextsubscript\textsubscript
\renewcommand{\textsubscript}[1]{%
  \ifdefstring{\f@shape}{it}
    {\kern-0.05em\oldtextsubscript{#1}\kern0.05em}
    {\oldtextsubscript{#1}}}
\makeatother

\let\oldthefootnote\thefootnote
\newcommand{\manualfootnote}[2]{%
  \let\thefootnote#1\footnotetext{#2}%
  \let\thefootnote\oldthefootnote}

\newcommand{\mathpar}[2]{\text{\parbox{#1}{\strut#2\strut}}}

\newcommand{\requireEmpty}[1]{%
  \ifempty{#1}{}{%
    \PackageError{mg1-gittins}
      {this macro needs an empty argument}
      {you passed a nonempty argument}}}

\newcommand{\unit}{\textsc}

\newcommand{\sysMa}[1]{\requireEmpty{#1}M/G\textsubscript{MP}/1}

\newcommand{\sys}[1]{\requireEmpty{#1}M\textsubscript{B}/G\textsubscript{MP}/1}

\newcommand{\X}{\mathbb{X}}
\newcommand{\preempt}{\mathsf{P}}
\newcommand{\nonpreempt}{\mathsf{NP}}
\newcommand{\Xnew}{X_{\mathsf{new}}}
\newcommand{\Xbatch}[1][]{\ifempty{#1}{\mathbf}{}X_{\mathsf{batch}\ifempty{#1}{}{, #1}}}
\newcommand{\xdone}{x_{\mathsf{done}}}
\newcommand{\Y}{\mathbb{Y}}
\newcommand{\Yopt}[1][r]{\Y^*\ifempty{#1}{}{\mkern-1.5mu(#1)}}

\newcommand{\hold}[1][x, r]{\mathop{}\!\mathsf{hold}\ifempty{#1}{}{(#1)}}
\newcommand{\Hold}[1][x, r]{\mathop{}\!\mathsf{Hold}\ifempty{#1}{}{(#1)}}
\newcommand{\serve}[1][x, r]{\mathop{}\!\mathsf{serve}\ifempty{#1}{}{(#1)}}
\newcommand{\Serve}[1][x, r]{\mathop{}\!\mathsf{Serve}\ifempty{#1}{}{(#1)}}
\newcommand{\game}[1][x, r]{\mathop{}\!\mathsf{game}\ifempty{#1}{}{(#1)}}
\newcommand{\rank}[1][x]{\mathop{}\!\mathsf{rank}\ifempty{#1}{}{(#1)}}
\newcommand{\ix}[1][x]{\mathop{}\!\mathsf{index}\ifempty{#1}{}{(#1)}}

\def\r-{$r$\=/}

\begin{document}

\title{The Gittins Policy in the M/G/1 Queue%
  \iftoggle{long}{\large\\[5pt](Extended and Revised Version)\vspace{-13pt}}{}}
\author{%
  \IEEEauthorblockN{Ziv Scully\makebox[0pt][l]{\textsuperscript{*}}}%
  \IEEEauthorblockA{%
    Carnegie Mellon University\\
    Pittsburgh, PA, USA\\
    zscully@cs.cmu.edu}%
  \and
  \IEEEauthorblockN{Mor Harchol-Balter\makebox[0pt][l]{\textsuperscript{*}}}%
  \IEEEauthorblockA{%
    Carnegie Mellon University\\
    Pittsburgh, PA, USA\\
    harchol@cs.cmu.edu}}
\maketitle

\iftoggle{long}{\thispagestyle{plain}\pagestyle{plain}}{}

\begin{abstract}
  The Gittins policy is a highly general scheduling policy
  that minimizes a wide variety of mean holding cost metrics in the M/G/1 queue.
  Perhaps most famously, Gittins minimizes mean response time in the M/G/1
  when jobs' service times are unknown to the scheduler.
  Gittins also minimizes weighted versions of mean response time.
  For example, the well-known ``$c\mu$~rule'',
  which minimizes class-weighted mean response time in the multiclass M/M/1,
  is a special case of Gittins.

  However, despite the extensive literature on Gittins in the M/G/1,
  it contains no fully general proof of Gittins's optimality.
  This is because Gittins was originally developed for
  the multi-armed bandit problem.
  Translating arguments from the multi-armed bandit to the M/G/1
  is technically demanding,
  so it has only been done rigorously in some special cases.
  The extent of Gittins's optimality in the M/G/1 is thus not entirely clear.

  In this work we provide the first fully general proof
  of Gittins's optimality in the M/G/1.
  The optimality result we obtain is even more general than was previously known.
  For example, we show that Gittins minimizes mean slowdown
  in the M/G/1 with unknown or partially known service times,
  and we show that Gittins's optimality holds under batch arrivals.
  Our proof uses a novel approach that works directly with the M/G/1,
  avoiding the difficulties of translating from the multi-armed bandit problem.
\end{abstract}

\section{Introduction}
\label{sec:intro}

\manualfootnote{*}{Supported by NSF grant nos. CMMI-1938909 and CSR-1763701
  and a Google Faculty Award.}

Scheduling to minimize mean holding cost in queueing systems
is an important problem.
Minimizing metrics such as mean response time,
weighted mean response time,
and mean slowdown
can all be viewed as special cases of minimizing holding cost
\citep{brumelle_relation_1971, heyman_relation_1980}.\footnote{%
  A job's \emph{response time} is the amount of time
  between its arrival and completion.
  Jobs may be sorted into \emph{classes}
  which are weighted by importance.
  A job's \emph{slowdown} is the ratio between
  its response time and service time.}
In single-server queueing systems,
specifically the M/G/1 and similar systems,
a number of scheduling policies
minimize mean holding cost in various special cases.
Two famous examples are the \emph{Shortest Remaining Processing Time (SRPT)} policy,
which minimizes mean response time
when service times are known to the scheduler,
and the ``$c\mu$~rule'',
which minimizes weighted mean response time
in the multiclass M/M/1 with unknown service times.

It turns out that there is a policy
that minimizes mean holding cost in the M/G/1
under very general conditions.
This policy, now known as the \emph{Gittins} policy
after one of its principal creators
\citep{gittins_multi-armed_1989, gittins_multi-armed_2011},
has a relatively simple form.
Gittins assigns each job an \emph{index},
which is a rating roughly corresponding to how valuable it would be
to serve that job.
A job's index depends only on its own state,
not the state of any other jobs.
Gittins then serves the job of maximal index at every decision time.
The ``magic'' of Gittins is in how it determines each job's index,
which we describe in detail in \cref{sec:gittins}.
SRPT and the $c\mu$~rule are both special cases of Gittins.

\begin{table}
  \setlength{\tabcolsep}{0pt}
  \centering
  \begin{threeparttable}
    \caption{Gittins Optimality Results for M/G/1-Like Queues}
    \label{tab:prior}
    \begin{tabular*}{\linewidth}{@{\extracolsep{\fill}}lllll@{}}
      \toprule
      \emph{Holding Cost}
      & \emph{Model}
      & \emph{Preemption}
      & \emph{Service Times}
      & \emph{Prior Proofs}\kern0.05em\tnotex{tnote:prior_proofs}
      \\
      \midrule
      All equal
      & M/G/1
      & allowed
      & known
      & \ref{ppf:schrage}
      \\
      & \sysMa{}
      & allowed
      & see \cref{sec:model}
      & \ref{ppf:scully}
      \\
      \midrule
      By class
      & M/G/1
      & allowed
      & unknown
      & \ref{ppf:sevcik_unknown}, \ref{ppf:gittins}
      \\
      & M/G/1+fbk\makebox[0pt][l]{\tnotex{tnote:feedback}}
      & not allowed
      & unknown
      & \ref{ppf:klimov}, \ref{ppf:achievable_region}
      \\
      & M/M/1+fbk\makebox[0pt][l]{\tnotex{tnote:feedback}}
      & allowed
      & unknown
      & \ref{ppf:lai}, \ref{ppf:achievable_region}, \ref{ppf:whittle}
      \\
      \midrule
      By class and
      & M/G/1
      & not allowed
      & known
      & \ref{ppf:fife}
      \\
      service time\makebox[0pt][l]{\tnotex{tnote:slowdown}}
      & M/G/1
      & allowed
      & known
      & \ref{ppf:sevcik_known}
      \\
      & M/G/1
      & allowed
      & unknown
      & \ref{ppf:von_olivier}
      \\
      \bottomrule
    \end{tabular*}
    \begin{tablenotes}
    \item
      \label{tnote:prior_proofs}
      A list of prior proofs appears in \cref{sec:prior}.
    \item
      \label{tnote:feedback}
      Here ``fbk'' stands for \emph{feedback},
      meaning that whenever a job exits the system,
      it has some probability of being replaced by another job.
    \item
      \label{tnote:slowdown}
      This includes minimizing mean \emph{slowdown},
      in which a job's holding cost is the reciprocal of its service time.
    \end{tablenotes}
  \end{threeparttable}
\end{table}

Given its generality,
it is perhaps unsurprising that the Gittins policy has been discovered several times.
Similarly, there are several proofs of its optimality under varying conditions.
\Cref{tab:prior} summarizes several M/G/1-like settings
in which Gittins has been studied,
and \cref{sec:prior} gives a more detailed overview.
In light of this substantial body of prior work,
the optimality of Gittins for minimizing mean holding cost in the M/G/1
is widely accepted in the literature
\citep{aalto_gittins_2009, aalto_properties_2011, hyytia_minimizing_2012,
  scully_soap_2018, scully_optimal_2018, scully_characterizing_2020}.

We ourselves are researchers
whose work often cites Gittins's optimality in the M/G/1.
However, in reviewing the literature,
we found that there is
\emph{no complete proof of Gittins's optimality in its full generality}.
This is in part because Gittins was originally developed not for the M/G/1
but for the Markovian multi-armed bandit problem
\citep{gittins_bandit_1979}.
There are elegant arguments for Gittins's optimality in the multi-armed bandit problem,
but they do not easily translate to the M/G/1.
Results for the M/G/1 thus suffer from a variety of limitations
(\cref{sec:prior}),
so the extent of Gittins's optimality in the M/G/1 is not entirely clear.

In this work, we give
\emph{a unifying presentation of the Gittins policy} in M/G/1-like systems,
resulting in the most general definition of Gittins (\cref{def:index})
and optimality theorem (\cref{thm:optimality}) to date.
Our approach deals directly with the M/G/1,
avoiding the difficulties of translating from the multi-armed bandit problem.
As a result, we actually \emph{extend the known scope} of Gittins's optimality,
such as including systems with batch arrivals.
We make the following contributions:
\begin{itemize}
\item
  We discuss many prior proofs of Gittins's optimality,
  detailing the limitations of each one
  (\cref{sec:prior}).
\item
  We give a new general definition of the Gittins policy
  (\cref{sec:gittins}).
  This involves introducing a new generalization of the M/G/1
  called the \emph{\sys{} queue}
  (\cref{sec:model}).
\item
  We state (\cref{sec:optimality_scope})
  and prove (\cref{sec:game, sec:optimality_proof})
  Gittins's optimality in the \sys{}.
\end{itemize}

\section{History of the Gittins Policy in the M/G/1}
\label{sec:prior}

In this section we review prior work on the Gittins policy
in M/G/1-like queues.
This includes work on special cases of Gittins,
such as SRPT in the case of known service times,
that are not typically thought of as instances of the Gittins policy.
Unfortunately,
every prior proof of Gittins's optimality is limited in some way.
Most limitations are one of the following:
\begin{enumerate}[(i)]
\item
  \label{limitation:finite}
  \emph{Job finiteness.}
  Most proofs assume some type of ``finiteness'' of the job model.
  This manifests as one of
  \begin{enumerate}[({i}-a)]
      \item
        \label{limitation:upper_bound}
        all service times being less than some finite bound,
      \item
        \label{limitation:finite_support}
        service time distributions being discrete with finitely many support points, or
      \item
        \label{limitation:finite_classes}
        finitely many job classes.
  \end{enumerate}
\item
  \emph{Simple job model or metric.}
  Some proof techniques that work for simple job models
  do not readily generalize.
  This includes models with
  \begin{enumerate}[({ii}-a)]
      \item
        \label{limitation:known}
        known service times,
      \item
        \label{limitation:exponential}
        unknown, exponentially distributed service times, or
      \item
        \label{limitation:nonpreemptive}
        unknown, generally distributed service times with nonpreemptive service.
  \end{enumerate}
\item
  \label{limitation:index_only}
  \emph{Only considers index policies.}
  Some proofs only show that Gittins is an optimal \emph{index policy},
  as opposed to optimal among all policies.
  An index policy is one that, like Gittins,
  assigns each job an index based on the job's state
  and always serves the job of maximum index.
\end{enumerate}
\iftoggle{long}{%
These limitations are significant because
they put some widely-believed results
on uncertain theoretical foundations.

A final limitation that applies to all prior proofs is that
although different jobs may have different holding costs,
\emph{each job's holding cost is constant.}
To the best of our knowledge,
ours is the first presentation of Gittins and proof of its optimality
that allows jobs' holding costs to change during service.}{}

\iftoggle{long}{%
  We now present prior work on the Gittins policy in rough chronological order,
  giving each decade a theme.
  The decades should be understood loosely,
  as the themes do not fit perfectly into decades.
  Throughout, when we refer to ``M/G/1 scheduling'',
  we mean the problem of minimizing mean holding cost in an M/G/1 queue or similar model.%
}{%
  We now present prior work on the Gittins policy in rough chronological order.
  Due to space limitations, we only briefly summarize most prior proofs.
}

\iftoggle{long}{%
  \subsection{1960s: Known Service Times}
}{%
  \subsection{1960s and 1970s: Initial Developments in Queueing}
}
\label{sec:prior:known}

\iftoggle{long}{%
The earliest results in M/G/1 scheduling
all featured known service times.
The most famous of these results is
the proof that SRPT minimizes mean response time \citep{schrage_proof_1968},
but researchers also made progress in systems
with variable holding costs and nonpreemptive service.}{}

\begin{priorproof}{\Citet{schrage_proof_1968}}
  \label{ppf:schrage}
\item[Model:]
  Preemptive single-server queue, known service times.
\item[Holding costs:]
  Same for all jobs.
\item[Limitations:]
  \cref{limitation:known}.
\end{priorproof}

\begin{priorproof}{\Citet[Section~4]{fife_scheduling_1965}}
  \label{ppf:fife}
\item[Model:]
  Nonpreemptive M/G/1, known service times.
\item[Holding costs:]
  Based on class and service time.
\item[Limitations:]
  \cref{limitation:finite_support, limitation:finite_classes,
    limitation:known, limitation:nonpreemptive}
\end{priorproof}

\begin{priorproof}{\Citet[Theorem~4-1]{sevcik_use_1971}}
  \label{ppf:sevcik_known}
\item[Model:]
  Preemptive M/G/1, known service times.
\item[Holding costs:]
  Based on class and service time.
\item[Limitations:]
  \cref{limitation:known, limitation:index_only}.
  \Citet[Conjecture~4-1]{sevcik_use_1971} argues informally
  that an index policy should be optimal.
\end{priorproof}

\iftoggle{long}{%
  \subsection{1970s: Unknown Service Times}
}{}
\label{sec:prior:unknown}

\iftoggle{long}{%
The next developments in M/G/1 scheduling
considered systems with unknown service times.
Handling the nonpreemptive M/G/1 is
not much harder than the known-service-time case
\citep[Section~5]{fife_scheduling_1965}.
Preemptive M/G/1 scheduling turns out to be a much harder problem.
\Citet{sevcik_use_1971} and \citet{olivier_kostenminimale_1972}
independently published solutions within a short time span,
though both proofs suffer from similar technical limitations.}{}

\begin{priorproof}{\Citet[Theorem~4-2]{sevcik_use_1971}}
  \label{ppf:sevcik_unknown}
\item[Model:]
  Preemptive M/G/1, unknown service times.
\item[Holding costs:]
  Based on class.
\item[Limitations:]
  \cref{limitation:finite_support, limitation:finite_classes, limitation:index_only}.
  \Citet[Conjecture~4-3]{sevcik_use_1971} argues informally
  that an index policy should be optimal.
\end{priorproof}

\begin{priorproof}{\Citet{olivier_kostenminimale_1972}}
  \label{ppf:von_olivier}
\item[Model:]
  Preemptive M/G/1, unknown service times.
\item[Holding costs:]
  Based on class and service time.
\item[Limitations:]
  \cref{limitation:finite_support, limitation:finite_classes, limitation:index_only}.
\end{priorproof}

One unique aspect of the \citet{olivier_kostenminimale_1972} result
deserves highlighting:
jobs' holding costs can depend on their unknown service times.
This allows minimizing metrics like mean slowdown
even when service times are unknown.
However, this result is not widely known in the queueing theory community,
perhaps in part because it has only been published in German.

\iftoggle{long}{%
  A partially preemptive M/G/1 problem was solved
  by \citet{klimov_time-sharing_1974},
  who studied a nonpreemptive M/G/1 with \emph{feedback},
  denoted M/G/1+fbk.
  In systems with feedback, whenever a job exits the system,
  it has some probability of immediately returning as another job,
  possibly of a different class.
  This model is partially preemptive in that a job returned to the system via feedback
  need not be served immediately.
  Another way of viewing systems with feedback
  is that each job is a discrete semi-Markov chain
  where each job class is a state.\footnote{%
    The word ``state'' here differs from our job model's terminology
    (\cref{sec:model:job}).
    In our terminology, each state in a semi-Markov chain
    corresponds to a connected set of states
    in a piecewise-deterministic Markov process.}
  \Citeauthor{klimov_time-sharing_1974}'s model is thus notable in that
  in addition to having unknown service times,
  jobs take stochastic paths through a state space.}{
  \citet{klimov_time-sharing_1974}
  studied a nonpreemptive M/G/1 with \emph{feedback},
  denoted M/G/1+fbk.
  In systems with feedback, whenever a job exits the system,
  it has some probability of immediately returning as another job,
  possibly of a different class.
}

\begin{priorproof}{\Citet{klimov_time-sharing_1974}}
  \label{ppf:klimov}
\item[Model:]
  Nonpreemptive M/G/1+fbk, unknown service times.
\item[Holding costs:]
  Based on class.
\item[Limitations:]
  \cref{limitation:finite_classes, limitation:nonpreemptive}.
\end{priorproof}

\iftoggle{long}{%
  \subsection{1980s: Connection to the Multi-Armed Bandit Problem}
}{
  \subsection{1980s and 1990s: Connection to Multi-Armed Bandits}
}
\label{sec:prior:bandit}

\iftoggle{long}{%
In parallel with the above developments in M/G/1 scheduling queues,
researchers were studying a seemingly very different stochastic control problem:
the \emph{multi-armed bandit} (MAB) problem.
The MAB was the setting in which the Gittins policy was first developed.
See \citet{gittins_bandit_1979} for a survey of early developments
and \citet{gittins_multi-armed_2011} for a modern overview.

At first glance, the MAB problem seems very different from M/G/1 scheduling.
\begin{itemize}
\item
  The MAB problem involves maximizing \emph{exponentially discounted} reward,
  whereas M/G/1 scheduling typically involve minimizing
  \emph{long-run average} costs.
\item
  The MAB problem involves a \emph{fixed set} of alternatives,
  whereas M/G/1 scheduling features a \emph{dynamically changing set} of jobs.
\end{itemize}
Nevertheless, as early as \citeyear{nash_optimal_1973},
\citet{nash_optimal_1973} showed that
a version of the MAB problem becomes M/G/1 scheduling
in the limit as the discount rate vanishes.
In the 1980s, several researchers further pursued in these ideas.
\Citet{lai_open_1988} reexamined work by \citet{klimov_time-sharing_1974}
on the nonpreemptive M/G/1+fbk, connected it to the MAB problem,
and extended it to the preemptive M/M/1+fbk.
\Citet{gittins_multi-armed_1989} extended work by \citet{nash_optimal_1973}
to continuous time.}{}

\begin{priorproof}{\Citet{lai_open_1988}}
  \label{ppf:lai}
\item[Model:]
  Preemptive M/M/1+fbk, unknown service times.
\item[Holding costs:]
  Based on class.
\item[Limitations:]
  \cref{limitation:finite_classes, limitation:exponential}.
\end{priorproof}

\begin{priorproof}{\Citet[Theorem~5.6]{gittins_multi-armed_1989}}
  \label{ppf:gittins}
\item[Model:]
  Preemptive M/G/1, unknown service times.
\item[Holding costs:]
  Based on class.
\item[Limitations:]
  \cref{limitation:upper_bound, limitation:finite_classes}.
\end{priorproof}

\Citeauthor{gittins_multi-armed_1989}'s result \citep{gittins_multi-armed_1989}
is often cited in the literature as proving
the Gittins policy's optimality in the M/G/1
\citep{aalto_gittins_2009, aalto_properties_2011, hyytia_minimizing_2012,
  scully_soap_2018, scully_optimal_2018, scully_characterizing_2020}.
As such, it deserves some more detailed discussion.

\Cref{ppf:gittins} has two main steps.
The first step simplifies the problem by assuming
the scheduler can only preempt jobs
in a discrete set of states\footnote{%
  In this setting, a job's state is the pair of
  its class and attained service.}
\citep[Theorem~3.28]{gittins_multi-armed_1989}.
The set can be countable in principle,
but the proof assumes a side condition
that is only guaranteed to hold if the set is finite.
\iftoggle{long}{%
  This side condition comes from translating a MAB result to the M/G/1,
  which involves taking the vanishing-discount limit.%
}{%
  The condition comes from translating
  multi-armed bandit results to the M/G/1.%
}

The second step uses a limit argument to allow unrestricted preemption
\citep[Theorem~5.6]{gittins_multi-armed_1989}.
However, because the first step is limited to finitely many job states,
the second step's result is also limited.
Specifically, it requires finitely many classes
and that all service times be less than some finite bound.
\iftoggle{long}{%
These limitations could be relaxed,
but only by checking the first step's side condition
for every system considered in the limit argument.}{}

\iftoggle{long}{%
  \subsection{1990s: Achievable Region Approach}
}{}
\label{sec:prior:achievable_region}

\iftoggle{long}{%
The \emph{achievable region approach}
was a new way of thinking about a wide variety of stochastic control problems,
including M/G/1 scheduling and the MAB problem.
\Citet{bertsimas_achievable_1995} and \citet{dacre_achievable_1999}
give surveys of the area.
While the achievable region method introduced important new ideas,
it did not extend the known scope of Gittins's optimality in M/G/1 scheduling.}{}

\begin{priorproof}{%
    Achievable region approaches\iftoggle{long}{
      \citep{bertsimas_achievable_1995, dacre_achievable_1999}%
    }{.
      See \citet{bertsimas_achievable_1995, dacre_achievable_1999},
      and references therein%
    }}
  \label{ppf:achievable_region}
\item[Model:]
  Preemptive M/M/1+fbk or nonpreemptive M/G/1+fbk,
  unknown service times.
\item[Holding costs:]
  Based on class.
\item[Limitations:]
  \cref{limitation:finite_classes, limitation:exponential, limitation:nonpreemptive}.
\end{priorproof}

\subsection{2000s and 2010s: Analyzing Gittins and Its Performance}
\label{sec:prior:analysis}

The 2000s and 2010s did not, for the most part,
see new proofs of Gittins's optimality.
Researchers instead studied properties of the Gittins policy
\citep{aalto_gittins_2009, aalto_properties_2011}
and analyzed its performance
\citep{whittle_tax_2005, hyytia_minimizing_2012, scully_soap_2018,
  scully_optimal_2018, scully_characterizing_2020}.
A performance analysis by \citet{whittle_tax_2005} based on dynamic programming
also resulted in an optimality proof,
but it did not expand the known scope of Gittins's optimality.

\begin{priorproof}{\Citet{whittle_tax_2005}}
  \label{ppf:whittle}
\item[Model:]
  Preemptive M/M/1+fbk, unknown service times.
\item[Holding costs:]
  Based on class.
\item[Limitations:]
  \cref{limitation:finite_classes, limitation:exponential}.
\end{priorproof}

\subsection{2020: Modeling Jobs as General Markov Processes}
\label{sec:prior:markov_process}

\iftoggle{long}{%
In 2020, \citet{scully_gittins_2020} studied minimizing mean response time
in the preemptive M/G/$k$,
showing that Gittins is near-optimal in a certain sense.
As a byproduct of their analysis of Gittins in the M/G/$k$,
they gave a new proof of Gittins's optimality in the M/G/1.
Their technique overcomes many limitations of prior proofs,
particularly limitation~\cref{limitation:finite},
but it applies only to the metric of mean response time.}{}

\begin{priorproof}{\Citet[Theorem~7.3]{scully_gittins_2020}}
  \label{ppf:scully}
\item[Model:]
  Preemptive \sysMa{},
  i.e. the preemptive \sys{} (\cref{sec:model}) without batch arrivals.
\item[Holding costs:]
  Same for all jobs.
\item[Limitations:]
  Assumes equal holding costs
  and that jobs are preemptible in any state.
\end{priorproof}

Our work can be seen as a significant extension of \cref{ppf:scully}.
Specific aspects we address that \citet{scully_gittins_2020} do not
include varying holding costs,
nonpreemptible or partially preemptible jobs,
and batch arrivals.

\section{System Model: the \sys{} Queue}
\label{sec:model}

We study scheduling in a generalization of the M/G/1 queue
to minimize a variety of mean holding cost metrics.
The average job arrival rate is~$\lambda$,
the service time distribution is~$S$,
and the load is $\rho = \lambda \E{S}$.
We assume $\rho < 1$ for stability.

We call our model the \emph{\sys{} queue}.
The ``M\textsubscript{B}'' indicates that jobs arrive in batches
with Poisson arrival times.
The ``G\textsubscript{MP}'' indicates generally distributed service times,
with each job's service time arising from an underlying Markov process.

The main feature of the \sys{} is that it models jobs as Markov processes.
The key intuition is:
\begin{quote}
  A job's \emph{state} encodes
  all information the scheduler knows about the job.
\end{quote}
This means that the job Markov process differs
depending on what information the scheduler knows.
For example, to model the perfect-information case
where the scheduler is told every job's service time when it arrives,
a job's state might be its remaining service time,
and the Markov process dynamics would be deterministic
(\cref{ex:known_service_time}).
On the other extreme,
if the scheduler knows nothing other than the overall service time distribution~$S$,
then a job's state might be the amount of service it has received so far,
and the Markov process dynamics would be stochastic
(\cref{ex:unknown_service_time}).
The \sys{} thus encompasses a wide variety of M/G/1-like queues.

This section explains the \sys{} queue in more detail.
The model's main feature is that the information the scheduler knows about a job
may change as the job receives service (\cref{sec:model:job}).
A job's preemptibility (\cref{sec:model:preempt})
and holding cost (\cref{sec:model:hold})
may also change during its service.

\subsection{Markov-Process Jobs}
\label{sec:model:job}

We model jobs as \emph{absorbing continuous-time strong Markov processes}.
The state of a job encodes all information that the scheduler knows about the job.
Without loss of generality,
we assume all jobs share a common state space~$\X$
and follow the same stochastic Markovian dynamics.
However, the realization of the dynamics may be different for each job.
In particular, the \emph{initial state} of each job
is drawn from a distribution~$\Xnew$,
so different jobs may start in different states.

While a job is in service,
its state stochastically advances according to the Markovian dynamics.
This evolution is independent of the arrival process
and the evolution of other jobs.
A job's state does not change while waiting in the queue.

In addition to the main job state space~$\X$,
there is one additional \emph{final state}, denoted~$\xdone$.
When a job enters state~$\xdone$, it completes and exits the system.
One can think of a service time~$S$
as the stochastic amount of time it takes for a job to go from
its initial state, which is drawn from~$\Xnew$, to the final state~$\xdone$.
Because we assume $\E{S} < \infty$,
every job eventually reaches~$\xdone$ with probability~$1$.
For ease of notation, we follow the convention that $\xdone \not\in \X$.

\begin{example}
  \label{ex:known_service_time}
  To model known service times,
  let a job's state be its remaining service time.
  The state space is $\X = (0, \infty)$,
  the initial state distribution~$\Xnew$ is the service time distribution~$S$,
  and the final state is $\xdone = 0$.
  During service, a job's state decreases at rate~$1$.
\end{example}

\begin{example}
  \label{ex:unknown_service_time}
  To model unknown service times,
  let a job's state be its attained service,
  meaning the amount of time it has been served so far.
  The state space is $\X = [0, \infty)$,
  all jobs start in initial state $\Xnew = 0$,
  and the final state $\xdone$ is an isolated point.
  During service, a job's state increases at rate~$1$,
  but it also has a chance to jump to~$\xdone$.
  The jump probability depends on the service time distribution~$S$:
  the probability a job jumps while being served
  from state~$x$ to state $y > x$ is $\P{S \leq y \given S > x}$.
\end{example}

\subsection{Preemptible and Nonpreemptible States}
\label{sec:model:preempt}

Every job state is either preemptible or nonpreemptible.
The job in service can only be preempted if it is in a preemptible state.
We write $\X_\preempt$ for the set of preemptible states
and $\X_\nonpreempt = \X \setminus \X_\preempt$
for the set of nonpreemptible states.
Naturally, we assume the scheduler knows which states are preemptible.

We assume all jobs start in a preemptible state,
i.e. $\Xnew \in \X_\preempt$ with probability~$1$.
This means that all jobs in the queue are in preemptible states,
and only the job in service can be in a nonpreemptible state.

We assume preemption occurs with no cost or delay.
Because a job's state only changes during service,
our model is preempt-resume, meaning that preemption does not cause loss of work.

\subsection{Batch Poisson Arrival Process}
\label{sec:model:arrival}

In the \sys{}, jobs arrive in batches.
We represent a batch as a list of states,
where the $i$th state is the initial state of the $i$th job in the batch.
The batch vector has distribution $\Xbatch = (\Xbatch[1], \dots, \Xbatch[B])$,
where $B$ is the distribution of the number of jobs per batch.
The batch arrival times are a Poisson process of rate $\lambda/\E{B}$,
with each batch drawn independently from~$\Xbatch$.
The initial state distribution $\Xnew$ is an aggregate distribution
determined by picking a random element from a length-biased sample of~$\Xbatch$.

We allow $\Xbatch$ to be an arbitrary distribution
over lists of preemptible states.
That is, the starting states of the jobs within a batch
can be correlated with each other or with the size of a batch.
However, after arrival, jobs' states evolve independently of each other
(\cref{sec:model:job}).

Our \sys{} model differs from the traditional M/G/1 with batch Poisson arrivals,
often denoted M\textsuperscript{X}/G/1, in an important way.
In the M\textsuperscript{X}/G/1,
service times within a batch are drawn i.i.d. from~$S$.
The \sys{} is more general in that
starting states within a batch can be correlated,
so service times within a batch can also be correlated.

\subsection{System State}
\label{sec:model:system}

The state of the system can be described by a list $(x_1, \dots, x_n)$.
Here $n$ is the number of jobs in the system,
and $x_i \in \X$ is the state of the $i$th job.
We denote the equilibrium distribution of the system state as $(X_1, \dots, X_N)$,
where $N$ is the equilibrium distribution of the number of jobs.

When discussing the equilibrium distribution of quantities
under multiple scheduling policies,
we use a superscript~$\pi$, as in~$N^\pi$,
to refer to the distribution under scheduling policy~$\pi$.

\subsection{Holding Costs and Objective}
\label{sec:model:hold}

We assume that there each job incurs a cost for each unit of time it is not complete.
Such a cost is called a \emph{holding cost}, and it applies to every job.
A job's holding cost depends on its state, so it may change during service.
We denote the holding cost of state $x \in \X$ by~$\hold[x]$.
Holding costs have dimension $\unit{cost}/\unit{time}$.
We assume that holding costs are deterministic, positive,\footnote{%
  The holding cost of nonpreemptible states
  does not impact minimizing mean holding cost
  (\cref{thm:hold_nonpreempt}),
  so one could have $\hold[x] \leq 0$ for $x \in \X_\nonpreempt$.}
and known to the scheduler.
For ease of notation, we also define $\hold[\xdone] = 0$.

Let $H = \sum_{i = 1}^N \hold[X_i]$ be the equilibrium distribution of
the total holding cost of all jobs in the system.
Our objective is to schedule to \emph{minimize mean holding cost~$\E{H}$}.

\subsection{What Does the Scheduler Know?}
\label{sec:model:known_to_scheduler}

The scheduler also knows, at every moment in time,
the current state of all jobs in the system.
This assumption is natural because the intuition of our model is that
a job's state encodes everything the scheduler knows about the job.

We assume the scheduler knows a description of the job model:
the state space~$\X$,
the subset of preemptible states $\X_\preempt \subseteq \X$,
and the Markovian dynamics that govern how a job's state evolves.
This assumption is necessary for the Gittins policy,
as the policy's definition depends on the job model.

Finally, we assume that the scheduler knows the holding cost $\hold[x]$
of each state $x \in \X$.
However, it is possible to transform some problems with unknown holding costs
into problems with known holding costs.
A notable example is minimizing mean slowdown when service times are unknown to the scheduler
(\cref{ex:slowdown}).
After transforming such problems into known-holding-cost form,
one can apply our results.

\subsection{Technical Foundations}
\label{sec:model:technical}

We have thus far avoided discussing
technical measurability conditions that the job model must satisfy.
For example, if the job Markov process has uncountable state space~$\X$,
one should make some topological assumptions on $\X$ and~$\X_\preempt$,
as well as some continuity assumptions on holding costs.
As another example, when discussing subsets $\Y \subseteq \X_\preempt$
(\cref{def:index, def:game}),
one should restrict attention to measurable subsets.
See \citet[Appendix~D]{scully_gittins_2020} for additional discussion.

We consider these technicalities outside the scope of this paper.
All of our results are predicated on being able
to apply basic optimal stopping theory
to solve the Gittins game (\cref{sec:game}).
Optimal stopping of general Markov processes is a broad field,
and the theory has been developed under many different types of assumptions
\citep{peskir_optimal_2006}.
Our main result (\cref{thm:optimality})
can be understood as proving Gittins's optimality
in any setting where optimal stopping theory of the Gittins game has been developed.

\section{The Gittins Policy}
\label{sec:gittins}

We now define the Gittins policy,
the scheduling policy that minimizes mean holding cost
in the \sys{} (\cref{sec:model}).

Before defining Gittins, we discuss its intuitive motivation.
Suppose we are scheduling with the goal of minimizing mean holding cost.
How do we decide which job to serve?
Because our objective is minimizing mean holding cost,
our aim should be to quickly lower the holding cost of jobs in the system.
We can lower a job's holding cost by completing it,
in which case its holding cost becomes $\hold[\xdone] = 0$,
or by serving it until it reaches a state with lower holding cost.

The basic idea of Gittins is to
always serve the job whose holding cost we can decrease the fastest.
To formalize this description, we need to define what it means
for a job's holding cost to decrease at a certain rate.

\subsection{Gittins Index}
\label{sec:gittins:index}

As a warm-up,
consider the setting of \cref{ex:known_service_time}:
the scheduler knows every job's service time,
and a job's state is its remaining service time.
Suppose that every state is preemptible.

How quickly can we decrease the holding cost of a job in state~$x$,
meaning $x$ remaining service time?
Serving a job from state~$x$ to state~$y$ takes $x - y$ time
and decreases the job's holding cost by $\hold[x] - \hold[y]$,
\iftoggle{long}{%
  which means
  \begin{equation*}
    \gp*{\mathpar{9em}{holding cost decrease rate from $x$ to~$y$}}
    = \frac{\hold[x] - \hold[y]}{x - y}.
  \end{equation*}%
}{%
  so the holding cost decreases at rate $(\hold[x] - \hold[y])/(x - y)$.%
}
To find the fastest possible decrease, we optimize over~$y$:
\begin{equation*}
  \gp*{\mathpar{9.5em}{maximum holding cost decrease rate from~$x$}}
  = \sup_{y \in [0, x)} \frac{\hold[x] - \hold[y]}{x - y}.
\end{equation*}
The above quantity is called the \emph{(Gittins) index} of state~$x$.
A state's index is the maximum rate at which we can decrease its holding cost
by serving it for some amount of time.

To generalize the above discussion to general job models,
we need to make two changes.
Firstly, because a job's state dynamics can be stochastic,
we need to consider serving it until it enters a \emph{set} of states~$\Y$.
Secondly, because we cannot stop serving a job while it is nonpreemptible,
we require $\Y \subseteq \X_\preempt$.

\begin{definition}
  \label{def:serve_hold}
  For all $x \in \X$ and $\Y \subseteq \X_\preempt$, let
  \begin{align*}
    \Serve[x, \Y] &= \gp*{\mathpar{14em}{%
      service needed for a job starting in state~$x$ to first enter $\Y \cup \{\xdone\}$}}, \\
    \serve[x, \Y] &= \E{\Serve[x, \Y]}, \\
    \Hold[x, \Y] &= \gp*{\mathpar{15.75em}{%
      holding cost of a job starting in state~$x$ when it first enters $\Y \cup \{\xdone\}$}}, \\
    \hold[x, \Y] &= \E{\Hold[x, \Y]}.
  \end{align*}
  To clarify, $\Serve[x, \Y]$ and $\Hold[x, \Y]$ are distributions.
  If $x \in \Y$, then $\Serve[x, \Y] = 0$ and $\Hold[x, \Y] = \hold[x]$.
\end{definition}

If we serve a job from state~$x$ until it enters~$\Y$,
its holding cost decreases at rate $(\hold[x] - \hold[x, \Y])/\serve[x, \Y]$
on average.
We obtain a state's Gittins index by optimizing over~$\Y$.

\begin{definition}
  \label{def:index}
  The \emph{(Gittins) index} of state $x \in \X$ is
  \begin{equation*}
    \ix = \sup_{\Y \subseteq \X_\preempt} \frac{\hold[x] - \hold[x, \Y]}{\serve[x, \Y]}.
  \end{equation*}
  When we say that a job has a certain index,
  we mean that the job's current state has that index.
\end{definition}

Given the definition of the Gittins index,
the Gittins policy boils down to one rule:
\iftoggle{long}{\begin{quote}
  A}{a}t every moment in time, unless the job in service is nonpreemptible,
  serve the job of \emph{maximal Gittins index},
  breaking ties arbitrarily.
\iftoggle{long}{\end{quote}}{}%

Because the Gittins index depends on the job model,
it might be more accurate to view Gittins not as one specific policy
but rather as a family of policies,
with one instance for every job model.
When we refer to ``the'' Gittins policy,
we mean the Gittins policy for the current system's job model.

\iftoggle{long}{%
\begin{example}[Gittins for mean response time]
  \label{ex:response_time}
  Consider the system from \cref{ex:unknown_service_time}.
  It has unknown service times,
  and a job's state~$x$ is its attained service.
  Suppose all states are preemptible.
  To minimize mean response time, we give all jobs holding cost~$1$.
  The Gittins index (\cref{def:index}) is then given by
  a formula well-known in the literature
  \citep{aalto_gittins_2009, aalto_properties_2011, scully_soap_2018}:
  \begin{equation*}
    \ix = \sup_{y > x} \frac{\P{S \leq y \given S > x}}{\E{\min\{S, y\} - x \given S > x}},
  \end{equation*}
\end{example}}{}

\subsection{Gittins Rank}
\label{sec:gittins:rank}

Some work on the Gittins policy
refers to the \emph{(Gittins) rank} of a state
\citep{sevcik_use_1971, sevcik_scheduling_1974, scully_soap_2018, scully_gittins_2020},
which is the reciprocal of its index:
\begin{equation*}
  \rank = \frac{1}{\ix}.
\end{equation*}
Gittins thus always serves the job of minimal rank.

The Gittins rank sometimes has a more intuitive interpretation than the Gittins index.
For instance, when jobs have known service times and constant holding cost~$1$,
a job's rank is its remaining service time,
and thus Gittins reduces to SRPT.

We use both the index and rank conventions in this work.
This section mostly uses the index convention.
\Cref{sec:game, sec:optimality_proof}, which prove Gittins's optimality,
use the rank convention because it better matches the authors' intuitions,
though this choice is certainly subjective.

\section{Scope of Gittins's Optimality}
\label{sec:optimality_scope}

Our main result is that Gittins is optimal in the \sys{}
with arbitrary state-based holding costs.
Specifically, Gittins is optimal among \emph{nonclairvoyant} scheduling policies,
which are policies that make scheduling decisions
based only on the current and past system states.

\begin{theorem}
  \label{thm:optimality}
  The Gittins policy minimizes mean holding cost in the \sys{}.
  That is, for all nonclairvoyant policies~$\pi$,
  \begin{equation*}
    \E{H^{\mathrm{Gittins}}} \leq \E{H^\pi}.
  \end{equation*}
\end{theorem}

All of the prior optimality results discussed in \cref{sec:prior}
are special cases of \cref{thm:optimality}.
This makes \cref{thm:optimality} a
\emph{unifying theorem for Gittins's optimality} in M/G/1-like systems.
\Cref{thm:optimality} also holds in scenarios not covered by any prior result.
For instance, no prior result handles batch arrivals
or holding costs that change during service.

\subsection{Mean Slowdown and Unknown Holding Costs}
\label{sec:optimality_scope:slowdown}

Recall from \cref{sec:model:hold} that we assume that
the holding cost of every job state is known to the scheduler.
However, some scheduling problems involve unknown holding costs.
An important example is minimizing mean slowdown,
in which a job's holding cost is the reciprocal of its service time.
Unless all service times are known to the scheduler,
this involves unknown holding costs.

Fortunately,
we can transform many problems with unknown holding costs
into problems with known holding costs.
Suppose a job's current unknown holding cost
depends only on its current and future states.
Then for all job states $x \in \X$, let
\begin{equation}
  \label{eq:hold_unknown}
  \hold[x]
  = \E*{\mathpar{9em}{unknown holding cost of a job in state~$x$}
      \given \mathpar{4.75em}{job reached state~$x$}},
\end{equation}
where the expectation is taken over a random realization of
a job's path through the state space.
The mean holding cost of nonclairvoyant policies
is unaffected by this transformation.

\begin{example}[Gittins for mean slowdown]
  \label{ex:slowdown}
  Consider the system from \cref{ex:unknown_service_time}.
  It has unknown service times,
  and a job's state~$x$ is its attained service.
  Suppose all states are preemptible.
  To minimize mean slowdown,
  we give a job with service time~$s$ holding cost~$s^{-1}$.
  This turns \cref{eq:hold_unknown} into
  \iftoggle{long}{%
    \begin{equation*}
      \hold[x] = \E{S^{-1} \given S > x},
    \end{equation*}%
  }{%
    $\hold[x] = \E{S^{-1} \given S > x}$,%
  }
  and the Gittins index\iftoggle{long}{ (\cref{def:index})}{} becomes
  \begin{equation*}
    \ix = \sup_{y > x} \frac{\E{S^{-1} \1(S \leq y) \given S > x}}{\E{\min\{S, y\} - x \given S > x}}.
  \end{equation*}
\end{example}

\section{The Gittins Game}
\label{sec:game}

In this section we introduce the \emph{Gittins game},
which is an optimization problem concerning a single job.
The Gittins game serves two purposes.
Firstly, it gives an alternative intuition for the Gittins rank.
Secondly, its properties are important for proving Gittins's optimality.
We define the Gittins game
(\cref{sec:game:def}),
study its properties,
(\cref{sec:game:shape, sec:game:give-up, sec:game:derivative}),
and explain its relationship to the Gittins rank
(\cref{sec:game:rank}).

\subsection{Defining the Gittins Game}
\label{sec:game:def}

The Gittins game is an optimal stopping problem concerning a single job.
We are given a job in some starting state $x \in \X$
and a \emph{penalty parameter} $r \geq 0$,
which has dimension $\unit{time}^2/\unit{cost}$.

The goal of the Gittins game is to end the game as soon as possible.
The game proceeds as follows.
\iftoggle{long}{\begin{itemize}
\item}{}%
  We begin by serving the job.
  The job's state evolves as usual during service (\cref{sec:model:job}).
  If the job completes, namely by reaching state~$\xdone$, the game ends immediately.
\iftoggle{long}{\item}{}%
  Whenever the job's state is preemptible, we may \emph{give up}.
  If we do so, we stop serving the job,
  and the game ends after deterministic delay $r \hold[y]$,
  where $y \in \X_\preempt$ is the job's state when we give up.
\iftoggle{long}{\end{itemize}}{\par}%
We assume the job's current state is always visible.
Playing the Gittins game thus boils down to deciding whether or not to give up
based on the job's current state.

Because the job's state evolution is Markovian,
the Gittins game is a Markovian optimal stopping problem.
This means there is an optimal policy of the following form:
for some \emph{give-up set} $\Y \subseteq \X_\preempt$,
give up when the job's state first enters~$\Y$.
The strong Markov property implies that this set $\Y$
need not depend on the starting state,
though it may depend on the penalty parameter.
We use this observation and \cref{def:serve_hold}
to formally define the Gittins game.

\begin{definition}
  \label{def:game}
  The \emph{Gittins game} is the following optimization problem.
  The parameters are a starting state $x \in \X$ and penalty parameter~$r$,
  and the control is a give-up set $\Y \subseteq \X_\preempt$.
  The \emph{cost of give-up set~$\Y$} is
  \begin{equation*}
    \game[x, r, \Y] = \serve[x, \Y] + r \hold[x, \Y].
  \end{equation*}
  The objective is to choose $\Y$ to minimize $\game[x, r, \Y]$.
  The \emph{optimal cost} or \emph{cost-to-go function} of the Gittins game is
  \begin{equation}
    \label{eq:game}
    \game = \inf_{\Y \subseteq \X_\preempt} \game[x, r, \Y].
  \end{equation}
\end{definition}





\subsection{Shape of the Cost-To-Go Function}
\label{sec:game:shape}

To gain some intuition for the Gittins game,
we begin by proving some properties of the cost-to-go function,
focusing on its behavior as the penalty parameter varies.

\begin{lemma}
  \label{thm:game_properties}
  For all $x \in \X$ and $r \geq 0$,
  the cost-to-go function $\game$ is
  \begin{enumerate\iftoggle{long}{}{*}}[(i), afterlabel={\nolinebreak}]
  \item
    \label{item:game_nondecreasing}
    nondecreasing in~$r$,
  \item
    \label{item:game_concave}
    concave in~$r$,
  \item
    \label{item:game_bound_hold}
    bounded by $\game \leq \serve[x, \X_\preempt] + r \hold[x, \X_\preempt]$,
  \item
    \label{item:game_bound_serve}
    bounded by $\game \leq \serve[x, \emptyset]$.
  \end{enumerate\iftoggle{long}{}{*}}
  When $x \in \X_\preempt\esub$,
  property~\cref{item:game_bound_hold} becomes $\game \leq r \hold[x]$.
\end{lemma}

\begin{proof}
  Properties~\cref{item:game_nondecreasing, item:game_concave}
  follow from~\cref{eq:game},
  which expresses $\game$ as an infimum of nondecreasing concave functions of~$r$.
  Properties~\cref{item:game_bound_hold, item:game_bound_serve}
  follow from the fact that two possible give-up sets
  are~$\X_\preempt$, meaning giving up as soon as possible,
  and~$\emptyset$, meaning never giving up.
  The simplification when $x \in \X_\preempt$ is due to \cref{def:serve_hold}.
\end{proof}



\subsection{Optimal Give-Up Set}
\label{sec:game:give-up}

We now characterize one possible solution to the Gittins game.
Because the Gittins game is a Markovian optimal stopping problem,
we never need to look back at past states when deciding when to give up.
This means we can find an optimal give-up set that depends only on
the penalty parameter~$r$.
We ask for each preemptible state:
is it optimal to give up immediately if we start in this state?
The set of states for which we answer yes is an optimal give-up set.

\begin{definition}
  \label{def:Yopt}
  The \emph{optimal give-up set} for the Gittins game with penalty parameter~$r$
  is
  \begin{equation*}
    \Yopt = \curlgp{x \in \X_\preempt \given \game = r \hold[x]}.
  \end{equation*}
  Noe that $\Yopt[0] = \X_\preempt$.
  We also let $\Yopt[\infty] = \emptyset$.
  For simplicity of language,
  we call $\Yopt$ ``the'' optimal give-up set,
  even though there may be other optimal give-up sets.
\end{definition}

Basic results in optimal stopping theory
\citep{peskir_optimal_2006}
imply that
\iftoggle{long}{\begin{equation*}}{\(}
  \game = \game[x, r, \Yopt]
\iftoggle{long}{,\end{equation*}}{\),}
so the infimum in \cref{eq:game} is always attained, namely by~$\Yopt$.

The sets $\Yopt$ are monotonic in~$r$,
i.e. $\Yopt \supseteq \Yopt[r']$ for all $r \leq r'$.
This is because increasing the penalty makes giving up less attractive,
so giving up is optimal in fewer states.

For most of the rest of this paper, when we discuss the Gittins game,
we consider strategies that use optimal give-up sets,
so we simplify the notation for that case.

\begin{definition}
  \label{def:serve_hold_abbrev}
  For all $x \in \X$ and $r \geq 0$, let
  \begin{equation*}
    \Serve = \Serve[x, \Yopt]
  \end{equation*}
  and similarly for $\serve$, $\Hold$, and $\hold$.
\end{definition}

\subsection{Derivative of the Cost-To-Go Function}
\label{sec:game:derivative}

Suppose we solve the Gittins game for penalty parameter~$r$,
then change the penalty parameter to $r \pm \epsilon$
for some small $\epsilon > 0$.
One would expect that the give-up set $\Yopt$
is nearly optimal for the new penalty parameter $r \pm \epsilon$,
which would imply $\game[x, r \pm \epsilon] \approx \serve + (r \pm \epsilon) \hold$.
One can use \cref{thm:game_properties}
and a classic envelope theorem \citep[Theorem~1]{milgrom_envelope_2002}
to formalize this argument.
\iftoggle{long}{}{%
For brevity, we omit the proof.
See \citet[Lemma~5.3]{scully_gittins_2020} for a similar proof.}

\begin{lemma}
  \label{thm:game_derivative}
  For all $x \in \X_\preempt\esub$,
  the function $r \mapsto \game$ is differentiable almost everywhere
  with derivative
  \begin{equation*}
    \dd{r} \game = \hold.
  \end{equation*}
\end{lemma}

\iftoggle{long}{%
For brevity, we omit the proof of \cref{thm:game_derivative}.
See \citet[Lemma~5.3]{scully_gittins_2020} for a similar proof.}{}

\subsection{Relationship to the Gittins Rank}
\label{sec:game:rank}

The Gittins game and the optimal give-up set
are closely related to the Gittins rank.
In fact, we can use the Gittins game to give an alternative definition
of a state's rank.\iftoggle{long}{

  \Cref{sec:game:def} describes the goal of the Gittins game
  as being to end the game as quickly as possible.
  An alternative intuition is that the goal is
  to reduce the job's holding cost to zero as quickly as possible.
  Under this intuition, we think of giving up as starting a process
  that decreases the job's holding cost at constant rate~$1/r$,
  where $r$ is the penalty parameter.
  Giving up in preemptible state~$x$ thus takes $r \hold[x]$ time.

  Consider playing the Gittins game
  with starting state $x \in \X_\preempt$ and penalty parameter $r \geq 0$.
  How do we decide whether or not to give up?
  That is, how do we determine whether $x \in \Yopt$?
  On one hand, \cref{def:index} tells us that by serving the job,
  we can decrease its holding cost at expected rate $1/\rank$.
  On the other hand, giving up decreases the holding cost at rate~$1/r$.
  The natural conclusion is that giving up is optimal,
  i.e. $x \in \Yopt$, if and only if $\rank \geq r$.

  The intuition above turns out to be exactly right:
  $\rank$ is the maximum penalty parameter~$r$ such that
  giving up is still optimal when in state~$x$.%
}{
  For brevity, we simply state the connection below.%
}

\begin{lemma}
  \label{thm:game_rank}
  \leavevmode
  \begin{enumerate}[(i), beginpenalty=10000]
  \item
    For all $r \geq 0$, we can write the optimal give-up set as%
    \iftoggle{long}{
      \begin{equation*}
        \Yopt = \curlgp{x \in \X_\preempt \given \rank \geq r}.
      \end{equation*}%
    }{\linebreak
      $\Yopt = \curlgp{x \in \X_\preempt \given \rank \geq r}$.%
    }
  \item
    For all $x \in \X_\preempt\esub$, we can write the Gittins rank of~$x$ as%
    \iftoggle{long}{
      \begin{align*}
        \rank
        &= \max\curlgp{r \geq 0 \given x \in \Yopt} \\
        &= \max\curlgp{r \geq 0 \given \game = r \hold[x]} \iftoggle{long}{\\
        &= \inf\curlgp{r \geq 0 \given x \not\in \Yopt} \\
        &= \inf\curlgp{r \geq 0 \given \game < r \hold[x]}}{}.
      \end{align*}%
    }{\linebreak
      $\rank = \max\curlgp{r \geq 0 \given x \in \Yopt}$.%
    }
  \end{enumerate}
\end{lemma}

\iftoggle{long}{%
For brevity, we omit the proof of \cref{thm:game_rank}.
See \citet[Lemma~5.4]{scully_gittins_2020} for a similar proof.}{}

\section{Proving Gittins's Optimality}
\label{sec:optimality_proof}

We now prove \cref{thm:optimality},
namely that Gittins minimizes mean holding cost in the \sys{}.
Our proof has four steps.
\iftoggle{long}{\begin{itemize}
\item}{}%
  We begin by showing that minimizing mean holding cost $\E{H}$
  is equivalent to minimizing the mean \emph{preemptible} holding cost $\E{H_\preempt}$,
  which only counts the holding costs of jobs in preemptible states
  (\cref{sec:optimality_proof:hold_split}).
\iftoggle{long}{\item}{}%
  We define a new quantity called \emph{\r-work},
  the amount of work in the system ``below rank~$r$''
  (\cref{sec:optimality_proof:r-work}).
\iftoggle{long}{\item}{}%
  We show how to relate an integral of \r-work
  to the preemptible holding cost~$H_\preempt$,
  (\cref{sec:optimality_proof:r-work_hold})
  with more \r-work implying higher holding cost.
\iftoggle{long}{\item}{}%
  We show that Gittins minimizes mean \r-work for all $r \geq 0$,
  so it also minimizes~$\E{H}$
  (\cref{sec:optimality_proof:r-work_min}).
\iftoggle{long}{\end{itemize}}{}%

\subsection{Preemptible and Nonpreemptible Holding Costs}
\label{sec:optimality_proof:hold_split}

\begin{definition}
  \label{def:hold_preempt_nonpreempt}
  The system's \emph{preemptible holding cost}
  is the total holding cost of all jobs in the system whose states are preemptible.
  It has equilibrium distribution
  \iftoggle{long}{%
    \begin{equation*}
      H_\preempt = \sum_{i = 1}^N \1(X_i \in \X_\preempt) \hold[X_i],
    \end{equation*}%
  }{%
    $H_\preempt = \sum_{i = 1}^N \1(X_i \in \X_\preempt) \hold[X_i]$,%
  }
  where $\1$ is the indicator function.
  \iftoggle{long}{%
    The \emph{nonpreemptible holding cost} is defined analogously
    and has equilibrium distribution
    \begin{equation*}
      H_\nonpreempt = \sum_{i = 1}^N \1(X_i \in \X_\nonpreempt) \hold[X_i].
    \end{equation*}%
  }{%
    The \emph{nonpreemptible holding cost} is defined analogously as
    $H_\nonpreempt = \sum_{i = 1}^N \1(X_i \in \X_\nonpreempt) \hold[X_i]$.
  }
\end{definition}

Our goal is to show that Gittins minimizes mean holding cost
$\E{H} = \E{H_\preempt} + \E{H_\nonpreempt}$.
The lemma below shows that
$\E{H_\nonpreempt}$ is unaffected by the scheduling policy.
Minimizing $\E{H}$ thus amounts to minimizing~$\E{H_\preempt}$.

\begin{lemma}
  \label{thm:hold_nonpreempt}
  In the \sys{},
  the mean nonpreemptible holding cost
  is the same under all scheduling policies:
  \begin{equation*}
    \E{H_\nonpreempt}
    = \lambda \E*{\mathpar{14.25em}{%
      total cost a job accrues while in a nonpreemptible state during service}}.
  \end{equation*}
\end{lemma}

\begin{proof}
  By a generalization of Little's law
  \iftoggle{long}{%
    \citep{brumelle_relation_1971, heyman_relation_1980}%
  }{%
    \citep{brumelle_relation_1971}%
  },
  \begin{equation*}
    \E{H_\nonpreempt}
    = \lambda \E*{\mathpar{11.5em}{%
      total cost a job accrues while in a nonpreemptible state}}.
  \end{equation*}
  The desired statement follows from the fact that
  if a job's state is nonpreemptible state, it must be in service
  (\cref{sec:model:preempt}).
\end{proof}

\subsection{Defining \r-Work}
\label{sec:optimality_proof:r-work}

\begin{definition}
  The \emph{(job) \r-work of state~$x$} is $\Serve$,
  namely the amount of service it requires
  to either complete or enter a preemptible state of rank at least~$r$.\footnote{%
    Strictly speaking,
    \cref{def:serve_hold, def:serve_hold_abbrev}
    introduce $\Serve$ as a distribution,
    so the \r-work of a job in state~$x$ is not $\Serve$ itself
    but rather a random variable with distribution $\Serve$.}
  The \emph{(system) \r-work} is
  the total \r-work of all jobs in the system.
  Its equilibrium distribution, denoted~$W(r)$, is
  \begin{equation*}
    W(r) = \sum_{i = 1}^N \Serve[X_i, r],
  \end{equation*}
  where $(X_1, \dots, X_N)$ is the equilibrium system state
  (\cref{sec:model:system}).
  We define the \emph{(system) preemptible \r-work}~$W_\preempt(r)$
  and \emph{(system) nonpreemptible \r-work}~$W_\nonpreempt(r)$
  similarly, except we only count \r-work from jobs
  in preemptible or nonpreemptible states:
  \iftoggle{long}{%
    \begin{align*}
      W_\preempt(r) &= \sum_{i = 1}^N \1(X_i \in \X_\preempt) \Serve[X_i, r], \\
      W_\nonpreempt(r) &= \sum_{i = 1}^N \1(X_i \in \X_\nonpreempt) \Serve[X_i, r].
    \end{align*}%
  }{%
    $W_\preempt(r) = \sum_{i = 1}^N \1(X_i \in \X_\preempt) \Serve[X_i, r]$
    and $W_\nonpreempt(r) = \sum_{i = 1}^N \1(X_i \in \X_\nonpreempt) \Serve[X_i, r]$.%
  }
\end{definition}

\begin{lemma}
  \label{thm:r-work_conditional_expectation}
  For all $r \geq 0$,
  \begin{equation*}
    \E{W_\preempt(r)} = \E*{\sum_{i = 1}^N \1(X_i \in \X_\preempt) \serve[X_i, r]}.
  \end{equation*}
\end{lemma}

\begin{proof}
  This follows from the law of total expectation
  and the fact that $\E{\Serve[X_i, r] \given X_i} = \serve[X_i, r]$.
\end{proof}

\subsection{Relating \r-Work to Holding Cost}
\label{sec:optimality_proof:r-work_hold}

\begin{theorem}
  \label{thm:r-work_hold}
  In the \sys{}, under all nonclairvoyant policies,
  \begin{equation*}
    \E{H_\preempt}
    = \int_0^\infty \frac{\E{W_\preempt(r)}}{r^2} \d{r}.
  \end{equation*}
\end{theorem}

\begin{proof}
  By \cref{thm:r-work_conditional_expectation, def:hold_preempt_nonpreempt},
  it suffices to show that for all $x \in \X_\preempt$,
  \begin{equation}
    \label{eq:r-work_hold_job}
    \hold[x] = \int_0^\infty \frac{\serve}{r^2} \d{r}.
  \end{equation}
  Using \cref{thm:game_derivative}, we compute
  \begin{equation*}
    \dd{r} \frac{\game}{r} = \frac{r \hold - \game}{r^2} = \frac{-\serve}{r^2}.
  \end{equation*}
  This means the integral in \cref{eq:r-work_hold_job}
  becomes a difference between two limits.
  \iftoggle{long}{%
    Using \cref{thm:game_derivative} for the $r \to 0$ limit
    and \cref{thm:game_properties}\cref{item:game_bound_serve}
    for the $r \to \infty$ limit,
    we obtain%
  }{%
    Using \cref{thm:game_derivative, thm:game_properties},
    we compute%
  }
  \begin{align*}
    \int_0^\infty \frac{\serve}{r^2} \d{r}
    &= \lim_{r \to 0} \frac{\game}{r} - \lim_{r \to \infty} \frac{\game}{r} \\
    &= \hold[x, 0] - 0
    \iftoggle{long}{\\ &}{}
    = \hold[x].
    \qedhere
  \end{align*}
\end{proof}

\Cref{thm:r-work_hold} implies that to minimize $\E{H_\preempt}$,
it suffices to minimize $\E{W_\preempt(r)} = \E{W(r)} - \E{W_\nonpreempt(r)}$ for all $r \geq 0$.
It turns out that $\E{W_\nonpreempt(r)}$, much like $\E{H_\nonpreempt}$,
is unaffected by the scheduling policy,
so it suffices to minimize $\E{W(r)}$.
\iftoggle{long}{}{%
We omit the proof,
as it is very similar to that of \cref{thm:hold_nonpreempt}.}

\begin{lemma}
  \label{thm:r-work_nonpreempt}
  In the \sys{}, the mean nonpreemptible \r-work $\E{W_\nonpreempt(r)}$
  is the same under all scheduling policies.
\end{lemma}

\iftoggle{long}{%
We omit the proof of \cref{thm:r-work_nonpreempt},
as it very similar to that of \cref{thm:hold_nonpreempt}.}{}

\subsection{Gittins Minimizes Mean \r-Work}
\label{sec:optimality_proof:r-work_min}

\Cref{thm:hold_nonpreempt, thm:r-work_hold, thm:r-work_nonpreempt}
together imply that if a scheduling policy
minimizes mean \r-work $\E{W(r)}$ for all $r \geq 0$,
then it minimizes mean holding cost~$\E{H}$.
We show that Gittins does exactly this, implying Gittins's optimality.

\begin{theorem}
  \label{thm:r-work_min}
  The Gittins policy minimizes mean \r-work in the \sys{}.
  That is, for all scheduling policies~$\pi$ and $r \geq 0$,
  \begin{equation*}
    \E{W^{\mathrm{Gittins}}(r)} \leq \E{W^\pi(r)}.
  \end{equation*}
\end{theorem}

Before proving \cref{thm:r-work_min},
we introduce the main ideas behind the proof.
For the rest of this section, fix arbitrary $r \geq 0$.
We classify jobs in the system into two types.
\begin{itemize}
\item
  A job is \emph{\r-good} if it is nonpreemptible
  or has Gittins rank less than~$r$,
  i.e. its state is in $\X \setminus \Yopt$.
\item
  A job is \emph{\r-bad} jobs if it has Gittins rank at least~$r$,
  i.e. its state is in $\Yopt$.
\end{itemize}
During service, a job may alternate between being \r-good and \r-bad.
Gittins minimizes \r-work because
the jobs that contribute to \r-work are exactly the \r-good jobs,
and Gittins always prioritizes \r-good jobs over \r-bad jobs.
This means that whenever the amount of \r-work in the system is positive,
Gittins decreases it at rate~$1$, which is as quickly as possible.

Given that Gittins decreases \r-work as quickly as possible,
does \cref{thm:r-work_min} immediately follow?
The answer is no: we need to look not just at how \r-work decreases
but also at how it increases.
Two types of events increase \r-work.
\begin{itemize}
\item
  Arrivals can add \r-work to the system.
\item
  \label{def:recycling}
  During service, a job can transition from being \r-bad to being \r-good
  as its state evolves.
  Using the terminology of \citet{scully_soap_2018, scully_gittins_2020},
  we say call this \emph{\r-recycling} the job.
  Every \r-recycling adds \r-work to the system.
\end{itemize}
Arrivals are outside of the scheduling policy's control,
but \r-recyclings occur at different times under different scheduling policies.
Because Gittins prioritizes \r-good jobs over \r-bad jobs,
all \r-recyclings occur when there is zero \r-work.
It turns out that because the batch arrival process is Poisson,
this \r-recycling timing minimizes mean \r-work.

\begin{proof}[Proof of \cref{thm:r-work_min}]
  We are comparing Gittins to an arbitrary scheduling policy~$\pi$.
  It is convenient to allow $\pi$ to be more powerful than an ordinary policy:
  we allow $\pi$ to devote infinite processing power to \r-bad jobs.
  This has two implications:
  \begin{itemize}
  \item
    Whenever there is \r-work in the system,
    $\pi$ controls at what rate it decreases,
    where $1$ is the maximum rate.
  \item
    Regardless of the rate at which \r-work is decreasing,
    whenever there is an \r-bad job in the system,
    $\pi$ controls at what moment in time it either completes or is \r-recycled.
  \end{itemize}
  A straightforward interchange argument shows that it suffices
  to only compare against policies~$\pi$ which are ``\r-work-conserving'',
  meaning they decrease \r-work at rate~$1$ whenever \r-work is nonzero.
  Gittins is also \r-work-conserving.

  It remains only to show that among \r-work-conserving policies,
  mean \r-work is minimized by only \r-recycling jobs when \r-work is zero.
  This follows from classic decomposition results
  for the M/G/1 with generalized vacations
  \citep{miyazawa_decomposition_1994}.
  We first explain how to view the \r-work in the \sys{} as
  the virtual work in a vacation system.\footnote{%
    \emph{Virtual work} in a vacation system
    is total remaining service time of all jobs in the system
    plus, if a vacation is in progress, remaining vacation time.}
  \begin{itemize}
  \item
    Interpret a batch adding $s$ \r-work to the \sys{}
    as an arrival of service time~$s$ in the vacation system.
  \item
    Interpret an \r-recycling adding $v$ \r-work to the \sys{}
    as a vacation of length~$v$ in the vacation system.
  \end{itemize}
  Using the above interpretation,
  a vacation system result
  of \citet[Theorem~3.3]{miyazawa_decomposition_1994} implies
  \begin{equation*}
    \E{W^\pi(r)} = c_1 + c_2 \E*{\mathpar{12em}{\r-work sampled immediately before $\pi$ \r-recycles a job}},
  \end{equation*}
  where $c_1$ and $c_2$ are constants
  that depend on the system parameters but not on the scheduling policy~$\pi$.
  Because Gittins prioritizes \r-good jobs over \r-bad jobs,
  Gittins only \r-recycles when \r-work is zero.
  This means the expectation on the right-hand side is zero under Gittins.
  But the expectation is nonnegative in general,
  so Gittins minimizes mean \r-work.
\end{proof}






\section{Conclusion}
\label{sec:conclusion}

We have given the first fully general statement (\cref{thm:optimality})
and proof of Gittins's optimality in the M/G/1.
This simultaneously improves upon, unifies, and generalizes prior proofs,
all which either apply only in special cases
or require limiting technical assumptions
(\cref{sec:prior}).

We believe Gittins's optimality holds
even more generally than we have shown.
For example, our proof likely generalizes to settings with ``branching'' jobs
or additional priority constraints on the scheduler
\citep[Section~4.7]{gittins_multi-armed_2011}.
It is also sometimes possible to strengthen the sense in which Gittins is optimal.
For example, SRPT is optimal for non-Poisson arrival times,
and Gittins sometimes stochastically minimizes holding cost
in addition to minimizing the mean.

\footnotesize
\bibliographystyle{IEEEtranN}
\bibliography{refs}

\end{document}